\documentclass[a4paper, 12pt]{article}
\usepackage[utf8]{inputenc}
\usepackage[english]{babel}
\usepackage{amssymb,amsmath,amsthm,eucal}
\usepackage[left=2cm, right=1.5cm, top=1cm, bottom=2cm]{geometry}

\usepackage{hyperref}
\hypersetup{
    colorlinks=true,
    linkcolor=blue,
    filecolor=magenta,      
    urlcolor=cyan,
}
\title{Gelfand-Tsetlin degeneration of shift of argument subalgebras in type D}
\author{Leonid Rybnikov and Mikhail Zavalin}
\begin{document}

\newtheorem{fact}{Fact}
\newtheorem{predl}{Theorem}
\newtheorem{prop}{Proposition}
\newtheorem{lemma}{Lemma}
\newtheorem{conj}{Conjecture}
\newtheorem{cor}{Corollary}
\newtheorem{rem}{Remark}
\newtheorem{defn}{Definition}

\maketitle
\begin{center}
{\it To our teacher Rafail Kalmanovich Gordin}
\end{center}

\medskip
\begin{abstract}
    The universal enveloping algebra of any semisimple Lie algebra $\mathfrak{g}$ contains a family of maximal commutative subalgebras, called shift of argument subalgebras, parametrized by regular Cartan elements of $\mathfrak{g}$. For $\mathfrak{g}=\mathfrak{gl}_n$ the Gelfand-Tsetlin commutative subalgebra in $U(\mathfrak{g})$ arises as some limit of subalgebras from this family. In our previous work \cite{RZ} we studied the analogous limit of shift of argument subalgebras for the Lie algebras $\mathfrak{g}=\mathfrak{sp}_{2n}$ and $\mathfrak{g}=\mathfrak{o}_{2n+1}$. We described the limit subalgebras in terms of Bethe subalgebras of twisted Yangians $Y^-(2)$ and $Y^+(2)$, respectively, and parametrized the eigenbases of these limit subalgebras in the finite dimensional irreducible highest weight representations by Gelfand-Tsetlin patterns of types C and B. In this note we state and prove similar results for the last case of classical Lie algebras, $\mathfrak{g}=\mathfrak{o}_{2n}$. We describe the limit shift of argument subalgebra in terms of the Bethe subalgebra in the twisted Yangian $Y^+(2)$ and give a natural indexing of its eigenbasis in any finite dimensional irreducible highest weight $\mathfrak{g}$-module by type D Gelfand-Tsetlin patterns.
\end{abstract}

\tableofcontents
\newpage

\section{Introduction}

\subsection{Commutative subalgebras in the universal enveloping algebra.}

The universal enveloping algebra of any semisimple Lie algebra $\mathfrak{g}$ contains a family of maximal commutative subalgebras, called shift of argument subalgebras $\mathcal{A}_{\mu}$, parametrized by regular Cartan elements $\mu$ of $\mathfrak{g}$. Their associated graded  $A_{\mu}:=\textrm{gr}\mathcal{A}_{\mu}\subset S\left(\mathfrak{g}\right)$ are subalgebras freely generated by the generators $S_i$ of centre $ZS\left(\mathfrak{g}\right)$ ($i=1,...,\textrm{rk}\mathfrak{g}$) and all their derivatives along $\mu$:
\begin{equation} \label{shift_arg_alg}
A_{\mu}:=<\partial^{k}_{\mu}S_i|i=1,...,\textrm{rk}\;\mathfrak{g};k=0,...,\textrm{deg}S_i-1>.
\end{equation}
The latter are Poisson-commutative subalgebras of maximal possible transcendence degree, first introduced by Mishchenko and Fomenko in \cite{MF}. From the above description it follows that the Poincar\`e series of such algebras have the following form:
\begin{equation}
\mathcal{P}_{\mathcal{A}_{\mu}}(x)=\mathcal{P}_{A_{\mu}}(x)=\prod\limits_{i=1}^{\textrm{rk}\mathfrak{g}}\prod\limits_{k=1}^{\textrm{deg}S_i}\frac{1}{1-x^k}.
\end{equation}

For $\mathfrak{g}=\mathfrak{gl}_n$ the Gelfand-Tsetlin commutative subalgebra in $U(\mathfrak{g})$ arises as some limit of subalgebras from this family (see \cite{V1,Sh}). Its eigenbasis in every finite dimensional irreducible $\mathfrak{g}$-module has a combinatorial description via the \textit{Gelfand-Tsetlin patterns} (or GT-patterns of type A). 

The analogous limits of shift of argument subalgebras exist for all classical Lie algebras. In $\cite{RZ}$ we studied such limit subalgebras for $\mathfrak{g}_n=\mathfrak{sp}_{2n}$ and $\mathfrak{g}_n=\mathfrak{o}_{2n+1}$. Namely,
we described the limit of the shift of argument subalgebra $\lim\limits_{\varepsilon\rightarrow0}\mathcal{A}_{\mu\left(\varepsilon\right)}$ for $\mu\left(\varepsilon\right)=F_{nn}+F_{n-1,n-1}\varepsilon+...+F_{11}\varepsilon^{n-1}\in\mathfrak{g}_n$ in terms of Bethe subalgebras $\mathcal{B}^-$ and $\mathcal{B}^+$ in the twisted Yangians $Y^-(2)$ and $Y^+(2)$, respectively (see Theorem~A of \emph{loc. cit.}). Here $F_{ij}$ is the basis of the classical Lie algebra $\mathfrak{g}_n$ expressed in the standard matrix units $E_{ij}$ as follows:
\begin{equation}
F_{ij}:=E_{ij}-sgn(i)sgn(j)E_{-j,-i}\;\;\textrm{for}\;i,j=-n,...,-1,1,...,n
\end{equation}
for $\mathfrak{g}_n=\mathfrak{sp}_{2n}$,
\begin{equation}
F_{ij}:=E_{ij}-E_{-j,-i}\;\;\textrm{for}\;i,j=-n,...,-1,1,...,n
\end{equation}
for $\mathfrak{g}_n=\mathfrak{o}_{2n}$, and
\begin{equation}
F_{ij}:=E_{ij}-E_{-j,-i}\;\textrm{for}\;i,j=-n,...,-1,0,1,...,n
\end{equation}
for $\mathfrak{g}_n=\mathfrak{o}_{2n+1}$.

Moreover, we showed that the spectrum of such limit subalgebra in any irreducible finite dimensional $\mathfrak{g}_n$-module $V_\lambda$ is indexed by the datum called Gelfand-Tsetlin patterns of types B and C (see Theorem~B of \emph{loc. cit}). In this note we prove the analogous assertions for classical Lie algebras of the (last remaining) type D, i.e. for the case of $\mathfrak{g}_n=\mathfrak{o}_{2n}$. 


\subsection{Twisted Yangian \texorpdfstring{$Y^{+}(2)$}{Lg}.} Let us recall the definition of the twisted Yangian $Y^+(2)$. For more detailed introduction to Yangians and their twisted versions we refer the reader to the classical paper by Molev, Nazarov, and Olshanski \cite{MNO}. 

\begin{defn}
The \textit{twisted Yangian} $Y^+(2)$ is an algebra generated by the coefficients $s_{ij}^{(r)}$ ($|i|=|j|=1$) of the power series
\begin{equation}
s_{ij}(u):=\sum\limits_{r\geq0}s_{ij}^{(r)}u^{-r}    
\end{equation} satisfying the following commutation relations:
\begin{eqnarray}\label{scruch3}
[s_{ij}(u),s_{kl}(v)]=\frac{1}{u-v}(s_{kj}(u)s_{il}(v)-s_{kj}(v)s_{il}(u))-\notag\\
-\frac{1}{u+v}\cdot(s_{i,-k}(u)s_{-j,l}(v)-s_{k,-i}(v)s_{-l,j}(u))+\\
+\frac{1}{u^2-v^2}(s_{k,-i}(u)s_{-j,l}(v)-s_{k,-i}(v)s_{-j,l}(u)).\notag\\
\notag
\end{eqnarray}
\begin{center}
and
\end{center}
\begin{equation}\label{scruch4}s_{-j,-i}(-u)=s_{ij}(u)+\frac{s_{ij}(u)-s_{ij}(-u)}{2u}
\end{equation}
for $|i|,|j|,|k|,|l|=1$.
\end{defn}

The twisted Yangian $Y^+(2)$ is a coideal subalgebra in the usual Yangianof the Lie algebra $\mathfrak{gl}_2$ $Y(2)=Y(\mathfrak{gl}_2)$. 

There is a distinguished maximal commutative subalgebra $\mathcal{B}^+\subset Y^+(2)$ called \emph{Bethe subalgebra}. The subalgebra $\mathcal{B}^+\subset Y^+(2)$ is generated by the elements $s_{11}^{(2m+1)}$ ($m\in\mathbb{Z}_{\geq0}$) and the center of $Y^+(2)$. For more general definition of Bethe subalgebras in twisted Yangians we refer the reader to $\cite{NO}$.
 
The Olshanski centralizer construction gives an algebra homomorphism $\varphi_i:Y^+(2)\to U\left(\mathfrak{o}_{2i}\right)^{\mathfrak{o}_{2i-2}}$ (see \cite[Theorem 2 (ii)]{RZ}, or \cite[Propositions 4.14 and 4.15]{MO}). We denote by $\mathcal{A}^+$ the commutative subalgebra of $U\left(\mathfrak{o}_{2n}\right)$ generated by $\cup_{i=1}^n\varphi_i(\mathcal{B}^+)$ and $\cup_{i=1}^nZU\left(\mathfrak{o}_{2i}\right)$. The first result of the present note is the following description for the limit quantum shift of argument subalgebra $\lim\limits_{\varepsilon\rightarrow0}\mathcal{A}_{\mu\left(\varepsilon\right)}$, generalizing \cite[Theorem~A]{RZ} to classical Lie algebras of the type $D$:
\begin{predl} \label{theorem_a}
The limit of the quantum shift of argument subalgebra $\lim\limits_{\varepsilon\rightarrow0}\mathcal{A}_{\mu\left(\varepsilon\right)}$ for $\mu\left(\varepsilon\right)=F_{nn}+F_{n-1,n-1}\varepsilon+...+F_{11}\varepsilon^{n-1}$ in the universal enveloping algebra $U\left(\mathfrak{o}_{2n}\right)$ coincides with $\mathcal{A}^+$.
\end{predl}

The idea of the proof is as follows. In \cite{RZ} we proved the inclusion $\mathcal{A}^+\subset\lim\limits_{\varepsilon\rightarrow0}\mathcal{A}_{\mu\left(\varepsilon\right)}$ for all classical types. Hence it remains to show that the Poincar\'e series of $\mathcal{A}^+$ is greater or equal to that of $\mathcal{A}_{\mu}$ for generic $\mu$. Following the same strategy as in \emph{loc. cit.} we get a lower bound for the Poincar\'e series of $\mathcal{A}^+$ by presenting some explicit set of elements in $\mathcal{A}^+$ which are algebraically independent and have the same degrees as the generators of $\mathcal{A}_{\mu}$. For the algebraic independence it is sufficient that the images of these elements in the associated graded algebra $\text{gr}\ \mathcal{A}^+\subset S(\mathfrak{g}_n)=\mathbb{C}[\mathfrak{g}_n]$ are algebraically independent. The latter is checked by taking the differentials of these elements at a regular nilpotent point $\tilde{e}\in \mathfrak{g}_n$ and proving that these differentials are linearly independent. The two main differences from the analogous computation in the types B and C are, first, that here we have to consider Pfaffians of the successively embedded $\mathfrak{o}_{2k}$ and compute their differentials, and, second, that contrary to the B and C cases, here the regular nilpotent element is \emph{not} a maximal Jordan block, so the computation is different. This computation occupies the most part of the next section. 

\subsection{\texorpdfstring{$Y^+(2)$}{Lg}-module structure on multiplicty spaces and the eigenbasis of \texorpdfstring{$\mathcal{A}^+$}{Lg}.}

According to the Olshansky centralizer construction there is an irreducible highest weight $\textrm{Y}^+(2)$-module structure on the multiplicity spaces of the restriction of any irreducible highest weight finite-dimensional $\mathfrak{o}_{2n}$-module $V_\lambda$ to $\mathfrak{o}_{2n-2}$. Let us recall the description of this $\textrm{Y}^+(2)$-module structure following \cite[Theorem 16 (ii)]{M}.

Let $L\left(\alpha,\beta\right)$ be the irreducible $\mathfrak{gl}_2$-module with the highest weight $(\alpha,\beta)$. Then we have a $Y(2)$-module structure on
\begin{equation}
L:=L(\alpha_1,\beta_1)\otimes L(\alpha_2,\beta_2)\otimes...\otimes L(\alpha_n,\beta_n)
\end{equation}
defined via the evaluation homomorphism $Y(2)\rightarrow U(\mathfrak{gl}_2)$ and the comultiplication on $Y(2)$. We also need the $1$-dimensional $Y^+(2)$-module $W(\delta)$ ($\delta\in\mathbb{C}$) defined as 
\begin{equation}\label{1_dim_repr_+}
s_{11}(u)\mapsto\frac{u+\delta}{u+1/2},\;\;\;s_{-1,-1}(u)\mapsto\frac{u-\delta+1}{u+1/2}w,
\end{equation}
and $s_{1,-1}(u)\mapsto0,\ s_{-1,1}(u)\mapsto0$, see \cite{M} for more details. Since $Y^+(2)$ is a coideal of the Hopf algebra $Y(2)$ (\cite{MNO}) the following module can be regarded as a $Y^+(2)$-module:
\begin{equation}
L'(\delta):=L\otimes W(\delta).
\end{equation}

\begin{prop} \label{mult_space_mod_struct}
Let $\lambda=\left(\lambda_1,...,\lambda_n\right)$ and $\mu=\left(\mu_1,...,\mu_{n-1}\right)$ be highest weights of finite-dimensional irreducible representations of $\mathfrak{o}_{2n}$ and $\mathfrak{o}_{2n-2}$ and $V_{\lambda}^{\mu}$ denote the corresponding multiplicity space. Then the action of $Y^+(2)$ on $V_{\lambda}^{\mu}$ defined as a composition of homomorphism $\varphi_n$ to $U\left(\mathfrak{o}_{2n}\right)^{\mathfrak{o}_{2n-2}}$ and a natural projection is irreducible and isomorphic to
\begin{equation}
L(\alpha_1,\beta_1)\otimes...\otimes L(\alpha_{n-1},\beta_{n-1})\otimes W(-\alpha_0),
\end{equation}
where $\alpha_1=\min\{-|\lambda_1|,-|\mu_1|\}-1/2$, $\alpha_0=\alpha_1+|\lambda_1+\mu_1|$,
$$\alpha_i=\min\{\lambda_i,\mu_i\}-i+1/2,\;\;i=2,...,n-1,$$
$$\beta_i=\max\{\lambda_{i+1},\mu_{i+1}\}-i+1/2,\;\;i=1,...,n-1.$$
\end{prop}


Consider the multiplicity space $V^{\mu}_{\lambda}:=\textrm{Hom}_{\mathfrak{g}_{n-1}}\left(V_{\mu},V_{\lambda}\right)$ Note that the above $Y^+(2)$ module \begin{equation}
L(\alpha_1,\beta_1)\otimes...\otimes L(\alpha_{n-1},\beta_{n-1})\otimes W(-\alpha_0),
\end{equation} has a (very naive) basis formed by the tensor products of the weight vectors in each tensor factor. This means that there is a basis of the multiplicity space indexed by collections of integers $\lambda'_1,\ldots,\lambda'_n$ satisfying the following inequalities:
$$-|\lambda_{1}|\geq\lambda'_{1}\geq\lambda_{2}\geq\lambda'_{2}\geq...\geq\lambda_{n-1}\leq\lambda'_{n-1}\geq\lambda_{n},$$
$$-|\mu_1|\geq\lambda'_{1}\geq\mu_{2}\geq\lambda'_{2}\geq...\geq\mu_{n-1}\geq\lambda'_{n-1}$$

Iterating the above restriction procedure (similarly to the construction of the Gelfand-Tsetlin basis) we get a basis of irreducible $\mathfrak{o}_{2n}$-module $V_\lambda$ corresponding to the combinatorial objects called \textit{D type pattern} $\Lambda$ for $\mathfrak{g}_n=\mathfrak{o}_{2n}$
$$\begin{matrix}
\lambda_{n1} && \lambda_{n2} && \ldots && \lambda_{nn}\\
& \lambda'_{n1} && \ldots  && \lambda'_{n-1,n-1}\\
\lambda_{n-1,1} && \ldots && \lambda_{n-1,n-1}\\
& \ldots && \ldots && \\
\lambda_{21} && \lambda_{22}\\
& \lambda'_{11}\\
\lambda_{11}
\end{matrix}$$
with $\lambda=(\lambda_{n1},...,\lambda_{nn})$ and all entries are simultaneously from $\mathbb{Z}_{\leq0}$ or $\{m+\frac{1}{2}|m\in\mathbb{Z},m+\frac{1}{2}\leq0\}$, satisfying the following set of inequalities
$$-|\lambda_{k1}|\geq\lambda'_{k-1,1}\geq\lambda_{k2}\geq\lambda'_{k-1,2}\geq...\geq\lambda_{k,k-1}\leq\lambda'_{k-1,k-1}\geq\lambda_{kk},$$
$$-|\lambda_{k-1,1}|\geq\lambda'_{k-1,1}\geq\lambda_{k-1,2}\geq\lambda'_{k-1,2}\geq...\geq\lambda_{k-1,k-1}\geq\lambda'_{k-1,k-1}$$
when $k=2,...,n$.

By \cite{KR} the spectrum of limit algebra $\mathcal{A}^+=\lim\limits_{\varepsilon\to0}\mathcal{A}_{\mu(\varepsilon)}$ in any irreducible finite-dimensional $\mathfrak{g}$-module $V_\lambda$ is simple. Moreover, from Theorem $\ref{theorem_a}$ we know that the limit algebra $\mathcal{A}^+$ is generated by commutative subalgebras $\mathcal{A}_k=\varphi_k(\mathcal{B}^+)$ in the successive centralizer algebras $U(\mathfrak{g}_k)^{\mathfrak{g}_{k-1}}$. Hence any eigenvector $v$ with respect to $\mathcal{A}^+$ in $V_\lambda$ has the following form: for some collection of highest weights $\lambda_k$ of $\mathfrak{g}_k$ (with $\lambda_n=\lambda$) have $v=\bigotimes\limits_{k=1}^nu_k$ where $u_k$ is an eigenvector of $\mathcal{B}^+$ in the multiplicity space $V^{\lambda_{k-1}}_{\lambda_k}=\textrm{Hom}_{\mathfrak{g}_{k-1}}\left(V_{\lambda_{k-1}},V_{\lambda_{k}}\right)$. Our purpose is to give a natural indexing of the (much more complicated) eigenbasis of $V_\lambda$ with respect to $\mathcal{A}^+$ by the same combinatorial data.

By Proposition~\ref{mult_space_mod_struct} we can regard the eigenbasis for $\mathcal{A}^+$ in $V_\lambda$ as an element of the continuous family of eigenbases for $\bigotimes\limits_{k=1}^n\mathcal{B}^+$ in the tensor products $\bigotimes\limits_{k=1}^n (\bigotimes\limits_{j=1}^{k-1}L(\alpha_{kj},\beta_{kj}))\otimes W(-\alpha_{k0})$ with $\alpha_{kj}$ being free parameters and the differences $\alpha_{kj}-\beta_{kj}$ being fixed integers. Then the following analog of \cite[Theorem B]{RZ} holds in the type $D$ case (without any changes in the proof). 

\begin{predl}
There is a path $\alpha(t)$ ($t\in[0,\infty)$) in the space of parameters $\alpha_{kj}$ such that \begin{itemize} \item $\alpha(0)$ is the collection of $\alpha_{kj}$ which occurs in our $\mathfrak{g}_n$-module $V_\lambda$;
\item the spectrum of $\bigotimes\limits_{k=1}^n\mathcal{B}^{+}$ on $\bigoplus\bigotimes\limits_{k=1}^n(\bigotimes\limits_{j=1}^{k-1}L(\alpha_{kj}(t),\beta_{kj}(t)))\otimes W(-\alpha_{k0})$ is simple for all $t>0$;
\item the limit of the corresponding eigenbasis as $t\to+\infty$ is just the product of the $\mathfrak{sl}_2$-weight bases in each $L(\alpha_{kj},\beta_{kj})=V_{\alpha_{kj}-\beta_{kj}}$.
\end{itemize}
\end{predl}

We reproduce the construction of such path $\alpha(t)$ in the end of the next section. Once we have such $\alpha(t)$ we can take the parallel transport of the eigenbasis at $t=0$ to the product of the weight bases at $t=\infty$ thus obtain a bijection between the former and the latter. Note that the product of the weight bases of $L(\alpha_{kj},\beta_{kj})$ is naturally enumerated by the GT-patterns of the type D. So from the above Theorem we get the following 

\begin{cor}
The spectrum of $\mathcal{A}^+$ in any finite-dimensional irreducible $\mathfrak{o}_{2n}$-module $V_{\lambda}$ of highest weight $\lambda$ can be naturally parametrized by the set of Gelfand-Tsetlin type D patterns.
\end{cor}

So the results of \cite{RZ} together with the present note can be summarized as follows:

\begin{predl}
Let $\mathfrak{g}_n$ be a Lie algebra of any classical type ($B_n$, $C_n$ or $D_n$). The limit of the shift of argument subalgebra $\lim\limits_{\varepsilon\rightarrow0}\mathcal{A}_{\mu\left(\varepsilon\right)}$ for $\mu\left(\varepsilon\right)=F_{nn}+F_{n-1,n-1}\varepsilon+...+F_{11}\varepsilon^{n-1}$ in the universal enveloping algebra $U\left(\mathfrak{g}_{n}\right)$ is $\mathcal{A}^\mp$. The spectrum of this subalgebra in any irreducible $\mathfrak{g}_n$-module $V_\lambda$ is simple and can be naturally identified with the set of Gelfand-Tsetlin patterns of the corresponding type. 
\end{predl}

The identification of the eigenbasis with the set of Gelfand-Tsetlin patterns depends on the choice of the path $\alpha(t)$. In the last section of this note we present some arguments towards this identification being independent of this choice. The detailed proof based on these arguments is the subject of our future work. We also expect that this indexing of the eigenbasis agrees with the $\mathfrak{g}_n$-crystal structure on the eigenbasis and on GT-patterns of classical types, defined in \cite{KR} and \cite{L}, respectively (see \cite[Conjecture~2]{RZ}). 

\subsection{Acknowledgements.} The work of both authors has been funded by the Russian Academic Excellence Project ’5-100’. The first
author has also been supported in part by the Simons Foundation.



\section{Proofs of the main theorems.}
\subsection{An auxiliary lemma on the Pfaffian polynomial.}

Recall that the determinant of a $2n\times2n$ skew-symmetric matrix $M$ is the square of a polynomial called the Pfaffian of $M$. This polynomial can be expressed in terms of matrix entries $M_{ij}$ as follows:
\begin{equation} \label{pfaffian_formula}
\textsf{pf}\left(M\right)=\frac{1}{2^nn!}\sum_{\sigma\in S_{2n}}m\left(\sigma\right),\\
\end{equation}
where $m\left(\sigma\right)=\textrm{sgn}\left(\sigma\right)M_{\sigma\left(1\right),\sigma\left(2\right)}\cdot...\cdot M_{\sigma\left(2n-1\right),\sigma\left(2n\right)}$. We identify $\{1,...,2n\}$ with $\{-n,...,-1,1,...,n\}$ by the unique order-preserving bijection and treat $S_{2n}$ as the set of all bijective maps from $\{1,...,2n\}$ to $\{-n,...,-1,1,...,n\}$.

Introduce the following skew-symmetric element:
\begin{equation}\label{skew-symm element}
\tilde{e}'=\sum_{i=1}^{n-1}\left(1\otimes M_{i,-i-1}+\left(-1\right)\otimes M_{-i-1,i}\right)+\left(1\otimes M_{-1,-2}+\left(-1\right)\otimes M_{-2,-1}\right).
\end{equation}

We have the following auxiliary result. 

\begin{lemma} \label{pfaff_lemma}
The differential of $\textsf{pf}(M)$ at $\tilde{e}'$ is proportional to $dM_{n,1}-dM_{n,-1}$ with a non-zero factor:
$$d_{\tilde{e}'}\textsf{pf}\;M\sim dM_{n,1}-dM_{n,-1}.$$
\end{lemma}

\begin{proof} For $\sigma\in S_{2n}$ we define a set $P(\sigma)$ of unordered pairs as:
\begin{equation}
P(\sigma):=\{\left(\sigma\left(2i-1\right),\sigma\left(2i\right)\right)|i=1,...,n\}.
\end{equation}

Note that $d_{\tilde{e}}m(\sigma)\neq 0$ iff $(i,-i-1)\in P(\sigma)$ for all $i=2,...,n-1$ and either $(n,1),(-1,-2)\in P(\sigma)$ or $(n,-1),(1,-2)\in P(\sigma)$. Hence we just need to examine the contribution to $d_{\tilde{e}'}\textsf{pf}M$ of all $m(\sigma)$ corresponding to $\sigma\in S_{2n}$ satisfying the conditions above.

Assume $(n,1)\in P(\sigma)$ and $\sigma$ satisfies the conditions above. Then we have $d_{\tilde{e}'}m(\sigma)\sim dM_{n,1}$. To describe such $\sigma$'s we need to fix the order in each pair of indices $(\sigma(2i-1),\sigma(2i))$ (note that the choice of the order does not change the value of $d_{\tilde{e}'}m(\sigma)$ since changing the order changes both the parity of the permutation $\sigma$ and the sign of $M_{\sigma(2i-1),\sigma(2i)}$). There is also an action of $S_n$ on $P(\sigma)$ (simply permuting the pairs) and this action preserves the property of satisfying the required condition. We observe that $d_{\tilde{e}'}m(\sigma)$ does not change if we act with a transposition on $P(\sigma)$. Indeed, both the parity of the permutation $\sigma$ and the monomial $m(\sigma)$ do not change. This shows that all the $d_{\tilde{e}'}m(\sigma)$'s with $(n,1)\in P(\sigma)$ are equal.

Multiplying by the transposition of 1 and -1 is a bijection between the sets of $\sigma$'s with $(n,1)\in P(\sigma)$ and with $(n,-1)\in P(\sigma)$. This means that each $d_{\tilde{e}'}m(\sigma)$ with $(n,-1)\in P(\sigma)$ has a contribution opposite to the one for $\sigma$'s with $(n,1)\in P(\sigma)$. This implies the assertion of the Lemma.
\end{proof}

\subsection{Algebraic independence of some elements of \texorpdfstring{$S\left(\mathfrak{o}_{2n}\right)$}{Lg}}

Introduce the principal nilpotent element of $\mathfrak{o}_{2n}$:
\begin{equation} \label{principal_nilp}
\tilde{e}=\sum_{i=1}^{n-1}\left(1\otimes F_{i,i+1}+\left(-1\right)\otimes F_{-i-1,-i}\right)+1\otimes F_{-1,2}+\left(-1\right)\otimes F_{-2,1}\in\mathfrak{o}_{2n}=\mathfrak{o}_{2n}^*.
\end{equation}

Let $x_{km}=\textrm{gr}\left(\textrm{Tr}\left(\mathcal{F}^{(k)}\right)^{2m}\right)$, $y_{km}=\textrm{gr}\left(\left[\left(\mathcal{F}^{(k)}\right)^{2m-1}\right]_{kk}\right)$, $p_k=\sqrt{\mathrm{det}\left(\left(F^{(k)}\right)^{2k}\right)}$, where
\begin{equation}
\mathcal{F}^{(k)}:=\sum\limits_{i,j=-k}^kF_{ij}\otimes E_{ij}\in U\left(\mathfrak{o}_{2n}\right)\otimes\textrm{End}\left(\mathbb{C}^{2k}\right),
\end{equation}
\begin{equation}
F^{(k)}:=\sum\limits_{i,j=-k}^kF_{ij}\otimes E_{ij}\in S\left(\mathfrak{o}_{2n}\right)\otimes\textrm{End}\left(\mathbb{C}^{2k}\right),
\end{equation}
and $k=1,...,n$; $m=1,...,k-1$.

Consider the following total order on the set $\tilde{U}=\{dF_{ij}|j=i,...,1,-1,...,-i\;\textrm{and}\;i=n,...,1\}$:
$$dF_{ij}\succ dF_{i'j'}\;\;\textrm{if}\;i>i'\;\textrm{or}\;i=i'\;\textrm{and}\;j<j'.$$

We extend this order to a partial order on the basis of $T_{\tilde e}^*\mathfrak{g}_n^*$ by simply setting each $dF_{ij}$ not from $\tilde{U}$ to be less than any element of $\tilde{U}$. Now we are ready to prove the following lemma.

\begin{lemma} \label{differentials_form}
The differentials of the elements $x_{nm}$, $y_{nm}$ ($m=1,...,n-1$) and $p_n$ at $\tilde{e}$ have the following form (here $\sim$ denotes proportionality):
$$dx_{n,n-1}|_{F=\tilde{e}}\sim dF_{n,-n+1}+(\textrm{linear combination of terms lower than}\;dF_{n,-n+1})$$
$$dy_{n,n-1}|_{F=\tilde{e}}\sim dF_{n,-n+2}+(\textrm{linear combination of terms lower than}\;dF_{n,-n+2})$$
$$dx_{n,n-2}|_{F=\tilde{e}}\sim dF_{n,-n+3}+(\textrm{linear combination of terms lower than}\;dF_{n,-n+3})$$
$$dy_{n,n-2}|_{F=\tilde{e}}\sim dF_{n,-n+4}+(\textrm{linear combination of terms lower than}\;dF_{n,-n+4})$$
$$...$$
if $n$ is even:
$$dx_{n,n/2}|_{F=\tilde{e}}\sim (dF_{n,-1}+dF_{n,1})+(\textrm{linear combination of terms lower than}\;dF_{n,1})$$
if $n$ is odd:
$$dy_{n,(n+1)/2}|_{F=\tilde{e}}\sim (dF_{n,-1}+dF_{n,1})+(\textrm{linear combination of terms lower than}\;dF_{n,1})$$
$$\;$$
$$dp_n|_{F=\tilde{e}}\sim (dF_{n,-1}-dF_{n,1})+(\textrm{linear combination of terms lower than}\;dF_{n,1})$$
$$...$$
$$dx_{n,1}|_{F=\tilde{e}}\sim dF_{n,n-1}+(\textrm{linear combination of terms lower than}\;dF_{n,n-1})$$
$$dy_{n,1}|_{F=\tilde{e}}\sim dF_{n,n}+(\textrm{linear combination of terms lower than}\;dF_{n,n})$$
\end{lemma}

\begin{proof}
Among all differentials $dx_{nm}|_{F=\tilde{e}}$, $dy_{nm}|_{F=\tilde{e}}$ and $dp_{n}|_{F=\tilde{e}}$ there is only $dx_{n,n-1}|_{F=\tilde{e}}$ which has a nonzero coefficient at $dF_{n,-n+1}$. The only monomials with a nonzero contribution to $dF_{n,-n+1}$ at $\tilde{e}$ are $$F_{n,-n+1}\cdot F_{-n+1,-n+2}\cdot...\cdot F_{-2,1}\cdot F_{1,2}\cdot...\cdot F_{n-1,n},$$ $$F_{n,-n+1}\cdot F_{-n+1,-n+2}\cdot...\cdot F_{-1,2}\cdot F_{2,3}\cdot...\cdot F_{n-1,n}.$$
Since their degrees are equal to $2n-2$ they can appear only in $x_{n,n-1}$. A similar situation occurs for $dF_{n-1,-n}$ and the corresponding monomials are:
$$F_{n-1,-n}\cdot F_{-n,-n+1}\cdot...\cdot F_{-2,1}\cdot F_{1,2}\cdot...\cdot F_{n-2,n-1},$$
$$F_{n-1,-n}\cdot F_{-n,-n+1}\cdot...\cdot F_{-2,-1}\cdot F_{-1,2}\cdot...\cdot F_{n-2,n-1}.$$
After writing down the differential at $\tilde{e}$ all of the above monomials have the same sign before $dF_{n,-n+1}$ (we make use of $dF_{n-1,-n}=-dF_{n,-n+1}$) hence the contribution is nonzero.

Analogously we prove all the statements till the point when they differ for odd and even values of $n$. In this proof we just consider the case of an even $n$ (the odd case can be proved identically). If we look at $dF_{n,1}$ then the only monomials which can contribute to it after differentiating at $\tilde{e}$ are:
$$F_{n,1}\cdot F_{1,2}\cdot...\cdot F_{n-1,n},$$
$$F_{-n,-n+1}\cdot...\cdot F_{-2,-1}\cdot F_{-1,-n}.$$
For the latter monomial we again notice $dF_{-1,-n}=-dF_{n,1}$. Since $n$ is even their differentials have the same sign at $dF_{n,1}$. Both of them require $m=n/2$ hence it must be $dx_{n,n/2}|_{F=\tilde{e}}$. But at the same time $dx_{n,n/2}|_{F=\tilde{e}}$ has a nonzero coefficient before $dF_{n,-1}$ because of the following monomials:
$$F_{n,-1}\cdot F_{-1,2}\cdot...\cdot F_{n-1,n},$$
$$F_{-n,-n+1}\cdot...\cdot F_{-2,1}\cdot F_{1,-n}.$$
It is easy to see that both $dF_{n,-1}$ and $dF_{n,1}$ appear in $dx_{n,n/2}|_{F=\tilde{e}}$ with the same nonzero coefficient.

Now we consider the Pfaffian-type element $p_n$. The $F$-matrix becomes skew-symmetric under the transformation 
\begin{equation}
F_{ij}\mapsto M_{i,-j}
\end{equation}
\begin{center}
and
\end{center}
\begin{equation}
p_n\mapsto\pm\textsf{pf}\left(M\right),\;\tilde{e}\mapsto\tilde{e}'
\end{equation}
where $M$ is the skew-symmetric image of $F$-matrix (the sign depends on the parity of $n$). From the Lemma $\ref{pfaff_lemma}$ it immediately follows that $dp_n|_{F=\tilde{e}}\sim\left(dF_{n,1}-dF_{n,-1}\right)$.

The remaining statements can be verified in a similar manner.
\end{proof}

This lemma implies the next statement.

\begin{lemma} \label{alg_indep}
The elements $x_{km}$, $y_{km}$ and $p_k$ for $k=1,...,n$ and $m=1,...,k-1$ are algebraically independent elements of $S\left(\mathfrak{o}_{2n}\right)$.
\end{lemma}

\begin{proof}
From the description of the differentials of these elements at the principal nilpotent element $\tilde{e}$ provided by Lemma $\ref{differentials_form}$ for $n,n-1,...,1$ we see that they are linearly independent, hence the Lemma follows.
\end{proof}

\subsection{The proof of Theorem 1.}
\begin{proof}
In Proposition 5 ($\cite{RZ}$) we showed that $\mathcal{A}^+\subset\lim\limits_{\varepsilon\rightarrow0}\mathcal{A}_{\mu\left(\varepsilon\right)}$. So it is sufficient to show that the subalgebras $\mathcal{A}^+$ and $\lim\limits_{\varepsilon\rightarrow0}\mathcal{A}_{\mu\left(\varepsilon\right)}$ have the same  Poincar\'e series. 

The Poincar\'e series for $\lim\limits_{\varepsilon\rightarrow0}\mathcal{A}_{\mu\left(\varepsilon\right)}$ coincides with such for $\mathcal{A}_{\mu}$ with generic $\mu\in\mathfrak{h}^{reg}$, hence, equals to:
\begin{equation}
P_{\mathcal{A}_{\mu\left(\varepsilon\right)}}\left(x\right)=\prod\limits_{k=1}^{n-1}\frac{1}{\left(1-x^{2k-1}\right)^{n-k}\left(1-x^{2k}\right)^{n-k}}\cdot\prod\limits_{k=1}^n\frac{1}{1-x^k}.
\end{equation}
On the other hand, the subalgebra $\mathcal{A}^+$ contains the elements $\varphi_k\left(s_{11}^{(2m-1)}\right)$ (of the degree $2m-1$), as well as the central generators $\mathrm{Tr}\left(\left(\mathcal{F}^{(k)}\right)^{2m}\right)$ (of the degree $2m$), for all $k=1,...,n$, $m=1,...,k$, and the symmetrization of the Pfaffians  $Symm(p_k)=Symm\left(\sqrt{\mathrm{det}\left(\left(F^{(k)}\right)^{2k}\right)}\right)$ (of the degree $k$) for all $k=1,...,n$ (here $Symm:S\left(\mathfrak{g}\right)\to U\left(\mathfrak{g}\right)$ is the symmetrization map from the symmetric algebra associated with $\mathfrak{g}$ to the universal enveloping algebra). So for proving the equality of the it remains to check that the above elements of $\mathcal{A}^+$ are algebraically independent. Moreover, it is sufficient to show that the images of these elements in the associated graded algebra $\text{gr}\ U\left(\mathfrak{o}_{2n}\right)=S\left(\mathfrak{o}_{2n}\right)$ are algebraically independent. Note that the latter follows from Lemma $\ref{alg_indep}$: indeed, we have $\textrm{gr}\left[\varphi_k\left(s_{11}^{(2m-1)}\right)\right]=y_{km}$, $\textrm{gr}\left[\mathrm{Tr}\left(\left(\mathcal{F}^{(k)}\right)^{2m}\right)\right]=x_{km}$ and $\textrm{gr}\left[Symm(p_k)\right]=p_k$.


\end{proof}

\subsection{The indexing of the spectrum by GT-patterns.}

First, we note that the proof of Theorem B in \cite{RZ} works for Theorem~2 of the present paper without any changes.

Recall that $Y^+(2)$-module $L'(z_1,...,z_n,\delta)$ is defined for $\alpha_1, ..., \alpha_n$, $\beta_1, ..., \beta_n$ as:
\begin{equation}
L'(z_1,...,z_n,\delta):=L(z_1+\alpha_1,z_1+\beta_1)\otimes...\otimes L(z_n+\alpha_n,z_n+\beta_n)\otimes W(\delta),
\end{equation}
where $z_i$'s are free parameters and $W(\delta)$ is the 1-dimensional representation of $Y^+(2)$ introduced before. Take $z_i=tu_i$ where $u_1<u_2<\ldots<u_n$ are fixed real numbers and $t\in\mathbb{R}_{>0}$. Let $\mathcal{B}^+(z_1,...,z_n)$ be the image of $\mathcal{B}^+$ in $\textrm{End}\left(L'(z_1,...,z_n,\delta)\right)$. The images of generators of $\mathcal{B}^+$ in the space of endomorphisms depend on $z_i$'s. We can replace the generators $s_{11}^{(2m+1)}$ ($m\in\mathbb{Z}_{\geq0}$) of $\mathcal{B}^+$ with $t^{-2m}s_{11}^{(2m+1)}$. The algebra generated by them still coincide with $\mathcal{B}^+(z_1,...,z_n)$ but now we can consider the desired limit by making $t\to+\infty$. Moreover, in \cite[Corollary 1]{RZ} we showed that the limit of the eigenbasis as $t\to+\infty$ ($z_i-z_j\to+\infty$ ($i>j$)) coincides with the product of weight bases for each $V_{\alpha_i-\beta_i}$ ($i=1,...,k$).

From the definition of $L'$ it follows that for every $\left(z_1,...,z_n\right)$ we have
\begin{equation}
L'\left(z_1,...,z_n,\delta\right)\simeq V_{\alpha_1-\beta_1}\otimes...\otimes V_{\alpha_n-\beta_n}.
\end{equation}
For generic $\left(z_1,...,z_n\right)$ subalgebra $\mathcal{B}^+$ acts on $L'\left(z_1,...,z_n,\delta\right)$ with simple spectrum.  On the other hand, the limit shift of argument subalgebra $\lim\limits_{\varepsilon\rightarrow0}\mathcal{A}_{\mu\left(\varepsilon\right)}=\mathcal{A}^+$ acts with simple spectrum on the multiplicity space $V_{\lambda_k}^{\lambda_{k-1}}$ (see \cite{KR}). The latter is identified with $L'(z_1,...,z_n,\delta)$ when all $z_i=0$.

Hence, by choosing an appropriate path connecting $z_i=0$ with the limit at $\infty$, the basis of $V_{\lambda^{(k)}}^{\lambda^{(k-1)}}$ can be naturally indexed by the collections of integers $\lambda_i'$ ($i=1,...,k$) such that $\alpha_i\geq\lambda_i'\geq\beta_i$. So it means that we have
$$0\geq\lambda_1'\geq\max\left(\lambda_1,\mu_1\right)$$
$$\min\left(\lambda_1,\mu_1\right)\geq\lambda_2'\geq\max\left(\lambda_2,\mu_2\right)$$
$$...$$
$$\min\left(\lambda_{k-1},\mu_{k-1}\right)\geq\lambda_k'\geq\lambda_k.$$
Continuing this restriction further we get the parametrization of the spectrum with Gelfand-Tsetlin patterns of type D.

\section{Discussion}
In \cite[Conjecture~1]{RZ} we conjectured that the indexing of the eigenvectors of $\mathcal{A}^\mp$ in $V_\lambda$ does not actually depend on the choice of the path $z(t)$ connecting $0$ to $\infty$. Here we present some speculation on this conjecture. 

The property that $\mathcal{B}^\mp$ has simple spectrum in $L'(z_1,\ldots,z_n,\delta)$ is in fact the intersection of the following two Zariski open conditions on $(z_1,\ldots,z_n)\in\mathbb{C}^n$:

\begin{enumerate}
    \item $\mathcal{B}^\mp$ has a cyclic vector in $L'(z_1,\ldots,z_n,\delta)$;
    \item $\mathcal{B}^\mp$ acts on $L'(z_1,\ldots,z_n,\delta)$ semisimply.
\end{enumerate}

For our conjecture, it is sufficient to show that the complement of this subset is of complex codimension~$2$ in some simply connected neighborhood of $\mathbb{R}_{\ge0}^n$ in $\mathbb{C}^n$. Indeed, this will mean that the spectrum of $\mathcal{B}^\mp$ in $L'(z_1,\ldots,z_n,\delta)$ is an unramified covering of a simply connected real manifold outside real codimension $4$, hence this covering is in fact trivial. So the parallel transport of the eigenvectors along any path connecting given $2$ points in $\mathbb{R}_{\ge0}^n$ (and avoiding the ``bad'' points) is the same.

First, one can show that for $z_i\in\mathbb{R}_{\ge0}$ the algebra $\mathcal{B}^\mp$ acts on $L'(z_1,\ldots,z_n,\delta)$ by normal operators with respect to some Hermitian scalar product. Hence the action of $\mathcal{B}^\mp$ on $L'(z_1,\ldots,z_n,\delta)$ is semisimple for all $(z_1,\ldots,z_n)$ in some simply connected neighborhood of $\mathbb{R}_{\ge0}^n$. 

Second, for the cyclic vector property, one can consider the nilpotent degeneration of $\mathcal{B}^\mp$ i.e. some $1$-parametric family of commutative subalgebras $\mathcal{B}^\mp(t)$ which are conjugate to each other for all nonzero $t$, and such that $\mathcal{B}^\mp(0)$ is generated by elements of the weight $-2$ with respect to the $\mathfrak{sl}_2$-weight grading. Once we have such degeneration, it is sufficient for the cyclic vector property that the highest vector of $L'(z_1,\ldots,z_n,\delta)$ is cyclic with respect to $\mathcal{B}^\mp(0)$. At the moment, we can construct such degeneration only in the symplectic case (i.e. only for $\mathcal{B}^-$), but we believe such degeneration exists in the orthogonal case as well. 

Next, we can apply the Shapovalov form technique to write explicitly the polynomial condition on $(z_,\ldots,z_n)$ expressing that the given weight subspace of $L'(z_1,\ldots,z_n,\delta)$ is not generated by $\mathcal{B}^\mp(0)$ from the highest vector. The result is the set of linear conditions $z_i-z_j=k,\ k\in\mathbb{Z}$, so at least outside this union of the hyperplanes $z_i-z_j=k$, the action of $\mathcal{B}^\mp$ on $L'(z_1,\ldots,z_n,\delta)$ is cyclic. On the other hand, each of these hyperplanes contain integer points where the action of $\mathcal{B}^\mp$ is the action of some limit shift of argument subalgebra on some multiplicity space, so the spectrum of $\mathcal{B}^\mp$ is simple by the general results of \cite{KR}. This means that the $\mathcal{B}^\mp$ has cyclic vector in $L'(z_1,\ldots,z_n,\delta)$ for generic $(z_1,\ldots,z_n)$ on each of the hyperplanes  $z_i-z_j=k$ as well. Hence the cyclic vector property holds for all $(z_1,\ldots,z_n)\in\mathbb{C}^n$ outside complex codimension $2$, so we are done.

We plan to give a detailed exposition of the above arguments in our future work.

\addcontentsline{toc}{section}{References}
\bibliographystyle{halpha}
\bibliography{D_case_bibliography}

\footnotesize{
{\bf L.R.}: National Research University
Higher School of Economics, Russian Federation,\\
Department of Mathematics, 6 Usacheva st, Moscow 119048;\\
Institute for Information Transmission Problems of RAS;\\
{\tt leo.rybnikov@gmail.com}}

\footnotesize{
{\bf M.Z.}: National Research University
Higher School of Economics, Russian Federation,\\
Department of Mathematics, 6 Usacheva st, Moscow 119048;\\
{\tt zavalin.academic@gmail.com}}

\end{document}